\documentclass[12pt]{amsart}
\usepackage{amsfonts}
\usepackage{epsfig}
\usepackage{graphicx}
\usepackage{amsmath}
\usepackage{amssymb}
\usepackage{latexsym}
\usepackage{rotating}
\usepackage{pstricks, pst-node, pst-text, pst-3d}
\usepackage{amsbsy}

\usepackage{amsthm}
\usepackage{mathrsfs}
\usepackage[all]{xy}
\usepackage{a4wide}

\input{amssym.def}
\newtheorem{theorem}{Theorem}
\newtheorem{proposition}{Proposition}

\newtheorem{lemma}{Lemma}
\newtheorem{remark}{Remark}

\theoremstyle{remark}

\newcommand{\re}{\text{\rm Re }}
\newcommand{\im}{\text{\rm Im }}
\newcommand{\sgn}{\text{\rm sgn }}

\newcommand{\I}{\text{\bf 1}}
\newcommand{\Hb}{\text{\bf H}}
\DeclareMathOperator{\SU}{SU}
\DeclareMathOperator{\U}{U}

\DeclareMathOperator{\spn}{span} 
 
\DeclareMathOperator{\tr}{tr} 
\DeclareMathOperator{\image}{im}

\DeclareMathOperator{\arccot}{arccot}

\newcommand{\s}{\vspace{0.3cm}}
\input epsf

\begin{document}
\title[Hopf fibration...]{Hopf fibration: geodesics and distances}
\author{Der-Chen Chang,\ Irina Markina, and Alexander Vasil'ev}

\address{Department of Mathematics, Georgetown University, Washington
D.C. 20057, USA}

\email{chang@georgetown.edu}

\thanks{}

\address{Department of Mathematics,
University of Bergen, P.O.~Box~7803, Bergen N-5008, Norway}

\email{irina.markina@uib.no} \email{alexander.vasiliev@uib.no}

\thanks{The first author is partially supported by a research
grant from the United State Army Research Office and a Hong Kong RGC
competitive earmarked research grant $\#$600607. The second and the
third authors have been  supported by the grant of the Norwegian Research Council \#177355/V30, by the NordForsk network `Analysis and Applications', grant \#080151, and by the European Science Foundation Research Networking Programme HCAA}




\date{}

\begin{abstract}
Here we study geodesics connecting two given points on odd-dimensional spheres respecting the Hopf fibration. This geodesic boundary value problem is completely solved in the case of 3-dimensional sphere and some partial results are obtained in the general case. The Carnot-Carath\'eodory distnce is calculated. We also present some motivations related to quantum mechanics.

\end{abstract}
\maketitle

\section{Introduction}

Sub-Riemannian geometry  is proved to play an important role in many applications, e.g., in
mathematical physics, geometric mechanics, robotics,  tomography, neurosystems, and control theory.
Sub-Riemannian geometry enjoys major differences from the Riemannian being a 
generalization of the latter at the same time, e.g., the notion of geodesic and length minimizer do not coincide even locally, the 
Hausdorff dimension
is larger than the manifold topological dimension, the exponential map is never a local 
diffeomorphism. There exists a large amount of literature developing sub-Riemannian 
geometry.  Typical general references are \cite{AS, CalinChang, Montgomery, Strichartz}.

The interest to odd-dimensional spheres comes first of all from finite dimensional quantum mechanics modeled over the Hilbert   space $\mathbb C^n$ where the dimension $n$ is the number of energy levels and  the normalized state vectors form the sphere $\mathbb S^{2n-1}\subset \mathbb C^n$. The problem of controlled quantum systems is basically the problem of controlled spin systems, which is reduced to the left- or right-invariant control problem on the Lie group $\SU(n)$. In other words,
these are problems of describing the sub-Riemannian structure of $\mathbb S^{2n-1}$ and the sub-Riemannian geodesics, see e.g., \cite{Ugo, Montgomery}. The special case $n=2$
is well studied and the sub-Riemannian structure is related to the classical Hopf fibration, see, e.g., \cite{Urbantke1, Urbantke2}. At the same time, the sub-Riemannian structure of $\mathbb S^3$ comes naturally from the non-commutative
group structure of $\SU(2)$ in the sense that two vector fields
span the smoothly varying distribution of the tangent bundle, and
their commutator generates the missing direction. The missing direction coincides with the Hopf vector field corresponding to the Hopf fibration. The sub-Riemannian geometry on $\mathbb S^3$ was studied in~\cite{CalinChangMarkina1, HR, GM1}, see also \cite{BosRossi}. Explicit formulas for  geodesics were obtained in \cite{ChMV} by
solving the corresponding Hamiltonian system, in \cite{HR} from a variational equation, in \cite{BosRossi} exploiting the Lie theory, in \cite{GM2} applying the structure of the principle $\mathbb S^1$-bundle. One of the important helping properties of odd-dimensional spheres is that there always exists
at least one globally defined non-vanishing vector field.

Observe that $\mathbb S^3$ is compact and many properties and results of sub-Riemannian geometry differ from the standard
nilpotent case, e.g., Heisenberg group or Engel group.  In the case $\mathbb S^{2n-1}$, $n>2$, we have no group structure and the main tool is the global action of the group $\U(1)$.  For example, in our paper we explicitly show that any two points of $\mathbb S^3$ can be connected with an  infinite number of geodesics. 

Because of important applications, we start our paper with the description of $n$-level quantum systems and motivation given by Berry phases. Further we continue with general formulas for geodesics. Then we concentrate our attention on the {\it geodesic boundary value problem} finding all sub-Riemannian geodesics between two given points.
In the case of $\mathbb S^{2n-1}$ we solve it for the points of the fiber and for $\mathbb S^3$ we solve it for arbitrary two points. The Carnot-Carath\'eodory distance is calculated.

\s
\noindent
{\bf Acknowledgement.} This work was initiated while the authors visited the National Center for Theoretical Sciences, Hsinchu, Taiwan R.O.C.  The authors acknowledge support and hospitality extended to them by the director of the center professor Wen-Ching Winnie Li  and the staff. 

\section{$n$-level quantum systems} The mathematical formulation of quantum mechanics is based on concepts of pure and mixed states. A complex separable Hilbert space $\mathfrak H$ with Hermitian product $\langle \cdot,\cdot\rangle$ (or $\langle \cdot|\cdot\rangle$ in Dirac notations) is called the {\it state space}. The exact nature of this Hilbert space depends on the concrete system. For an $n$-level quantum system, $\mathfrak H=\mathbb C^n$ with the standard Hermitian product $\langle z|w\rangle=\sum_{j=1}^nz_j\bar{w}_j$. An {\it observable} is  a self-adjoint linear operator acting on the state space. A {\it state} $\rho$ is a special case of observable  which  is Hermitian $\rho^{\dagger}=\rho$, normalized by $\tr \rho=1$, and positive $\langle w|\rho|w\rangle \geq 0$ for all vector $|w\rangle \in \mathfrak H$, where $\rho|w\rangle$ denotes the result of the action of the operator $\rho$  on the vector $|w\rangle$. A {\it pure state} is the one-complex-dimensional projection operator $\rho=|z\rangle\langle z|$ onto the vector $|z\rangle \in \mathfrak H$, i.e., an operator satisfying $\rho^2=\rho$. Other states are called {\it mixed states}.

The space of pure states is isomorphic to the projectivization $\mathfrak H_{\mathbb P}$ of the Hilbert space $\mathfrak H$. So equivalently we can define pure states  as normalized vectors $\langle z|z\rangle=1$ modulo a complex scalar $e^{i\theta}$, where $\theta$ is called a phase.    In the case of the $n$-level quantum system $\mathfrak H=\mathbb C^n$, normalization and the phase factor allow us to represent the space of pure states as the complex projective  space $\mathbb C{\mathbb P}^{n-1}\equiv \mathbb C^n_{\mathbb P}$. The second operation of phase factorization is realized by the higher-dimensional Hopf fibration
\[
\mathbb S^1 \hookrightarrow \mathbb S^{2n-1}\to \mathbb C{\mathbb P}^{n-1}, 
\]
where $\mathbb S^{2n-1}$ is the real $(2n-1)$-dimensional sphere embedded into $\mathbb C^n$, and the base space   $\mathbb C{\mathbb P}^{n-1}$ is the set of orbits of the action $\mathbb S^1$ on $\mathbb S^{2n-1}$. 

In what follows, the pure states $\rho$ are  the elements of  the projective space $\mathbb C{\mathbb P}^{n-1}$, at the same time we use the notation of state  $|z\rangle\in \mathbb S^{2n-1}$ to refer to a normalized representative of the phase-equivalence class $\rho$. The real dimension of the space of pure states is $2n-2$. The projective space $ \mathbb C{\mathbb P}^{n-1}$ endowed with
the  K\"ahlerian  Fubini-Study  metric  on $ \mathbb C{\mathbb P}^{n-1}$ becomes a metric space, in which this Riemannian metric is given by the real part of the Fubini-Study metric and it coincides with the push-forward of  the standard Riemannian metric (the real part of the Hermitian one) on the $(2n-1)$-sphere by the Hopf projection. 

 In the case $n=2$ the projective complex plane $\mathbb C{\mathbb P}$ is isomorphic to the sphere $\mathbb S^2$ (called the Bloch sphere for 2-level systems  in physics), which is thought of as the set of orbits of the classical Hopf fibration
 \[
\mathbb S^1 \hookrightarrow \mathbb S^{3}\to \mathbb S^2.
\]
 Each pair of antipodal points on the Bloch sphere corresponds to a mutually exclusive pair of states of the particle, namely, spin up and spin down. The Bloch sphere and the Hopf fibration describe the topological structure of a quantum mechanical two-level system, see \cite{Urbantke1, Urbantke2}. The interest to two-level systems, an old subject, recently has gained a renewed interest due to recent progress in quantum information theory and quantum computation, where two-level quantum systems became qubits coupled in $q$-registers. A qubit state is represented up to its phase by a point on the Bloch sphere. The topology of a pair of entangled two-level systems is given by the Hopf fibration
 \[
\mathbb S^3 \hookrightarrow \mathbb S^{7}\to \mathbb H{\mathbb P}, 
\]
where $\mathbb H$ is the space of quaternions and $\mathbb H{\mathbb P}\cong \mathbb S^4$ is its projectivization, see \cite{Mosseri}. Generally, for entangled $n$-level systems we have
 \[
\mathbb S^3 \hookrightarrow \mathbb S^{4n-1}\to \mathbb H{\mathbb P}^{n-1}. 
\]
The underlying manifold for the Lie group $\SU(2)$ is $\mathbb S^{3}$.
Considering the higher  dimensional group $\SU(n)$, we see that it acts 
on $\mathbb S^{2n-1}$. However its dimension is $n^2-1>2n-1$, $n>2$, and its manifold only contains the invariant sphere $\mathbb S^{2n-1}$. Returning back to the information theory motivation, the relevant group for $p$-qubits is $\SU(2^p)$, see, e.g., \cite{Bou, qubit}. 

Let us now concentrate on $n$-level systems.    Time evolution of a quantum system in the Schr\"odinger picture assumes time-dependent states
while observables are conserved. 
Time evolution of a quantum state is described by the Schr\"odinger equation
\[
\frac{d}{dt}|z(t)\rangle=\frac{-i}{\hbar}H|z(t)\rangle,
\]
where $H$ is a special observable, the Hamiltonian, a Hermitian operator corresponding to the total energy of the system. Its spectrum is the set of possible outcomes of energy measurements which we suppose to be discrete and non-degenerate.  The normalization condition is preserved by time evolution of the quantum system. The Bloch equation for the projection operator corresponding to the state $|z(t)\rangle$  becomes the Hamiltonian equation
\[
\frac{d}{dt}\rho(t)=\frac{-i}{\hbar}[H, \rho],
\]
where $[H, \rho]$ stands for the operator commutator. Observe that the state $|z\rangle$ is viewed as an element of $\mathbb S^{2n-1}$ while the operator $\rho$ is thought of as an element of $\mathbb C{\mathbb P}^{n-1}$. Thus the solutions to the Schr\"odinger equation describe trajectories
on $\mathbb S^{2n-1}$, and the solution to the Bloch equation describe trajectories on $\mathbb C{\mathbb P}^{n-1}$. The solution to the Schr\"odinger equation with the Hamiltonian $H(t)$,  $-iH(t)\in\mathfrak s\mathfrak u(n)$, is of the form
$|z(t)\rangle=g(t)|z(0)\rangle$, where $g(t)\in \SU(n)$. 
Observe that if $H$ is an Hermitian operator and if we denote by $h_0=\frac{1}{n}(\tr H -1)$, then $\Hb=H-h_0\I$ is an Hermitian operator and $\tr (H-h_0\I)=1$. The operator $H_0=h_0\I$ is called the {\it free Hamiltonian}. If the system is considered under some external fields (e.g.,  an electromagnetic field) whose amplitudes are represented	by some	functions	$h_1(t), . . . , h_m(t)$, then $H(t)=H_0+\sum_{k=1}^m h_k(t)\Hb_k$, where $\Hb_k$ are self-adjoint operators representing the coupling between the system and the external fields. One of the topical problems is {\it population transfer}: to find the amplitudes $h_k(t)$, $k=1\dots m$, such that the system starts at an eigenstate $|z_0\rangle$
of the free Hamiltonian $H_0$ and ends at time $T$ at another eigenstate $|z_l\rangle$ of $H_0$

The controllability problem resides in the so-called bracket-generating condition by Chow and Rashevsky \cite{Chow, Rashevsky}. The system is controllable if and only if the Lie hull
\[
{\rm Lie\ }(iH_0,iH_1,\dots, iH_m)=\mathfrak s\mathfrak u(n).
\]

Suppose that the operator $\rho(t)$, being viewed as an element of the projective space, evolves as a loop in time $t\in [0,T]$, $\rho(0)=\rho(T)$. A  crucial concept in many quantum mechanical effects is so-called the {\it Berry phase} (or Berry-Pancharatnam following \cite{Berry, Pancharatnam}).  The state $|z(t)\rangle$ of the system evolves according to the  Schr\"odinger equation. Suppose that $|z(T)\rangle=e^{i\Theta}|z(0)\rangle$. Then $\Theta=\Theta_{geom}+\Theta_{dyn}$, where $\Theta_{geom}$ is the Berry phase and $\Theta_{dyn}$ stands for the standard dynamical phase given by 
\[
\Theta_{dyn}=-\frac{1}{\hbar}\int_0^T\langle z(t)|H|z(t)\rangle dt.
\]
Observe that the instantaneous energy $\langle z(t)|H|z(t)\rangle$ is real because $H$ is self-adjoint. The geometric Berry phase refers to the curvature of the bundle.
It emphasizes that the single-time probabilistic description does not exhaust the physical content of quantum mechanics and appears because of the bundle structure of the
space of states (see, e.g., \cite{Anast}). As Berry \cite{Berry} mentions, Barry Simon commented in 1983 that the geometric phase factor has an interpretation in terms of holonomy
because the phase curvature two-form emerges naturally as the first Chern class of a Hermitian line bundle.

Let us suppose that the Hamiltonian $H$ is perturbed $H_{\varepsilon(t)}$
by a varying parameter $\varepsilon(t)$, $\varepsilon(0)=\varepsilon(T)$ slowly enough so that the adiabatic approximation is applied with $T$ large. The systems undergoes
transport along the closed path in the parametric space $\varepsilon$.  The {\it adiabatic theorem} states that  if the system starts at an  eigenstate of the Hamiltonian $H_{\varepsilon(0)}$, then it will end in the corresponding eigenstate of the Hamiltonian $H_{\varepsilon(T)}$ at the final point  under small perturbation remaining in its instantaneous
 eigenstate in all intermediate points $t\in [0,T]$. This theorem has been known from early years of quantum mechanics due to Max Born and Vladimir Fock.

Once the controllability problem is solved one can look for a more efficient way to construct loops in the projective space of pure states. A typical way is to minimize the energy
of transfer. The corresponding  optimal control problem is equivalent to finding sub-Riemannian geodesics on $\mathbb S^{2n-1}$. 
The precise meaning of this will be explained in the following section.

\section{Sub-Riemannian structure and geodesics on $\mathbb S^{2n-1}$}

\subsection{Hopf fibration}

We consider the $(2n-1)$-dimensional sphere $\mathbb S^{2n-1}$ as a real manifold  embedded into the complex space $\mathbb C^n$. The Riemannian metric $d$ on the sphere is the restriction of the Euclidean metric from $\mathbb R^{2n}$, that is, the real part of the Hermitian product $\langle\cdot,\cdot\rangle$ in $\mathbb C^n$.
The complex Hopf fibration $h$ is a principal circle bundle 
$$\mathbb S^1\ \hookrightarrow\  \mathbb S^{2n-1}\quad{\overset{h}\longrightarrow}\quad\mathbb C\mathbb P^{n-1}$$ 
over the complex projective space $\mathbb C\mathbb P^{n-1}$.  This roughly means that the preimage $h^{-1}(m)$ of any point $m\in \mathbb C\mathbb P^{n-1}$ is isomorphic to the unit circle $\mathbb S^1$ endowed with the Lie group structure of $\U(1)$. The set $h^{-1}(m)\subset \mathbb S^{2n-1}$ is called the {\it fiber}, and $\mathbb C\mathbb P^{n-1}$ is the {\it base space}, see, e.g., \cite{CalinChang, Montgomery}.

The intersection of a complex line with the unit sphere $\mathbb S^{2n-1}$ gives a unit circle $\mathbb S^1$ that is equipped with the group structure of $\U(1)$. The action of $e^{{2\pi it}}\in \U(1)$ on $ \mathbb S^{2n-1}$ is defined by $$e^{{2\pi it}}z=e^{{2\pi it}}(z_1,\ldots, z_n)=(e^{{2\pi it}}z_1,\ldots,e^{{2\pi it}}z_n).$$  We will use both real $(x_1,y_1,\ldots, x_n,y_n)$ and complex coordinates $z_k=x_k+iy_k$, $k=1,\ldots,n$.  

The horizontal tangent space $\mathcal{H}T_z$ at $z\in \mathbb S^{2n-1}$  is the maximal complex subspace of the real tangent space $T_z$. The unit normal real vector field $N(z)$ at $z\in \mathbb S^{2n-1}$ is given by $$N(z)=\sum_{k=1}^{n}(x_k\partial_{x_k}+ y_k\partial _{y_k})=2\re\sum_{k=1}^{n}z_k\partial_{z_k}.$$ The real vector field $$V(z)=iN(z)=\sum_{k=1}^{n}(-y_k\partial_{x_k}+ x_k\partial _{y_k})=2\re\sum_{k=1}^{n}iz_k\partial_{z_k}$$ is globally defined and non-vanishing and we call it the {\it vertical vector field}. The vector space $\spn_{\mathbb R}\{N(z),V(z)\}$ is a two dimensional subspace of $\mathbb C^n$ that inherits the standard almost complex structure from~$\mathbb C^n$, in other words we have an isomorphism $$\spn_{\mathbb R}\{N(z),V(z)\}\cong\spn_{\mathbb C}\{N(z)+iN(z)\}$$ of vector spaces. The orthogonal complement to  $\spn_{\mathbb C}\{N(z)+iN(z)\}$ with respect to the Hermitian product is exactly the horizontal tangent space $\mathcal{H}T_z\subset T_z$. If we denote by $\varphi(t)=e^{{2\pi it}}z$, $t\in[0,1)$, the fiber at a point $z\in \mathbb S^{2n-1}$, then it is easy to see that $\dot\varphi(t)=2\pi V(\varphi(t))$. Thus, $\spn_{\mathbb R}\{V(z)\}$ is the tangent space of the fiber $\varphi$ at $z=\varphi(t)$, that we denote by $\mathcal FT_z$. The vertical tangent bundle of $\varphi$ we denote by $\mathcal FT$, and the horizontal tangent bundle is denoted by $\mathcal HT$. Then, 
\begin{equation}\label{orthdecom}
T_z=\mathcal{H}T_z\oplus \mathcal FT_z,
\end{equation} 
where the direct sum means the orthogonal decomposition with respect to the Riemannian metric $d$ at each $z\in \mathbb S^{2n-1}$. The restriction of $d$ to the horizontal and vertical subspaces is denoted by $d_{\mathcal H}$ and $d_{\mathcal F}$, respectively. 

The group $\SU(n)$ acting in $\mathbb C^n$ leaves $\mathbb S^{2n-1}$ invariant and preserves the orthogonal decomposition~\eqref{orthdecom}.

Concluding the above construction we will work with a sub-Riemannian manifold that is the triple $(\mathbb S^{2n-1}, \mathcal HT, d_{\mathcal H})$ consisting of the smooth manifold $\mathbb S^{2n-1}$, the horizontal sub-bundle (or distribution $ \mathcal HT$), and the sub-Riemannian metric $d_{\mathcal H}$.
The same horizontal bundle can be obtained by considering the standard C-R geometry of odd-dimensional spheres or thinking of $\mathbb S^{2n-1}$ as a contact manifold endowed with the one-form~\cite{GM1, GM2}
\begin{equation}\label{contactf}
\mu=\sum_{k=1}^{n}(-y_k\,dx_k+ x_k\,dy_k).
\end{equation}

We are looking for sub-Riemannian geodesics, or energy minimizing curves $\gamma(t)$ on $\mathbb S^{2n-1}$. Since the curves have to belong to the sphere and to be horizontal, we conclude that the tangent vector $\dot \gamma(t)$ is orthogonal to complex line $\spn_{\mathbb C}(N+iN)$. Indeed, 
\begin{equation}\label{horcond}
\re\langle \gamma,\dot \gamma\rangle= 0, \quad\re\langle i\gamma,\dot \gamma\rangle=- \im\langle \gamma,\dot \gamma\rangle=0\Longrightarrow\langle \dot \gamma,\gamma\rangle=0,\qquad \forall\,\,t\in [0,1].
\end{equation}

\subsection{Equations of geodesics} 

The Hamiltonian function in sub-Riemanian geometry is defined by using the notion of co-metric~\cite{Montgomery, Strichartz}. As it was mentioned above, the sub-Riemannian metric $d_{\mathcal H}$ is defined by the restriction of the Euclidean metric from $\mathbb R^{2n}$ to the horizontal sub-bundle $\mathcal HT$. The restriction of the Euclidean metric from $\mathbb R^{2n}$ to fibers $\mathcal FT$ gives the metric $d_{\mathcal F}$ with respect of which the vertical vector field $V$ has length one. 
Given $d_{\mathcal H}$ we may define a linear mapping $g_{\mathcal H}(z):T^*_z\to T_z$. This linear map $g_{\mathcal H}$ is uniquely defined by the following conditions:
\begin{itemize}
\item[1.]{$\image(g_{\mathcal H}(z))=\mathcal HT_z$ at any point $z\in \mathbb S^{2n-1}$},
\item[2.]{$p(v)=d_{\mathcal H}(g_{\mathcal H}(z)[p],v)$ for $v\in \mathcal HT_z$, $p\in T^*_z$ at each $z\in \mathbb S^{2n-1}$.}
\end{itemize} We denote by $T^{\bot}_z$ the one-dimensional kernel of $g_{\mathcal H}(z)$ at $z\in \mathbb S^{2n-1}$. The generator of the kernel is the contact form or connection one-form~\eqref{contactf}. The linear map $g_{\mathcal H}(z)$ is positively definite on $T^*_z/T^{\bot}_z $ and symmetric: $p(g_{\mathcal H}(z)[\omega])=d_{\mathcal H}(g_{\mathcal H}(z)[p],g_{\mathcal H}(z)[\omega])=\omega(g_{\mathcal H}(z)[p])$,  $\ p,\omega\in T^*_z$, $z\in \mathbb S^{2n-1}$. The map $g_{\mathcal H}$ is called co-metric.

Analogously, we define a linear map $g_{\mathcal F}(z):T^*_z\to T_z$ by 
\begin{itemize}
\item[1.]{$\image(g_{\mathcal F}(z))=\mathcal FT_z$ at any point $z\in \mathbb S^{2n-1}$}
\item[2.]{$p(w)=d_{\mathcal F}(g_{\mathcal F}(z)[p],w)$ for $w\in \mathcal FT_z$, $p\in T^*_z$ at each $z\in \mathbb S^{2n-1}$.}
\end{itemize} 

Now we are able to set up two Hamiltonian functions: the horizontal Hamiltonian function $H_{\mathcal H}(z,p)$ and the vertical one $H_{\mathcal F}(z,p)$, $(z,p)\in T^*$. We write $$H_{\mathcal H}(z,p)=\frac{1}{2}p(g_{\mathcal H}(z)[p]),\qquad
H_{\mathcal F}(z,p)=\frac{1}{2}p(g_{\mathcal F}(z)[p]).$$

\begin{lemma}
The horizontal $H_{\mathcal H}$ and vertical $H_{\mathcal F}$ Hamiltonian functions are Poisson involutive: $$\{H_{\mathcal H},H_{\mathcal F}\}=0.$$
\end{lemma}

\begin{proof} Consider a neighborhood in $T^*\mathbb S^{2n-1}$ constructed by means of a local trivialization of the circle bundle $$\mathbb S^1\ \ \hookrightarrow\  \ \mathbb S^{2n-1}\quad{\overset{h}\longrightarrow}\quad\mathbb C\mathbb P^{n-1}.$$ Let $U$ be an open neighborhood in $\mathbb C\mathbb P^{n-1}$, $W=h^{-1}(U)$, and $\psi$ is a local trivialization map $\psi:W\to U\times \mathbb S^1$. The linear map $\psi^*(z)$ generated by $\psi$ defines a diffeomorphism on cotangent bundles $T^*W\cong T^*U\times T^*\mathbb S^1$. 

Notice that the decomposition $TU\times T\mathbb S^1$ is orthogonal with respect to the metric $d$, and the metrics $d_{\mathcal H}$ and $d_{\mathcal F}$ were defined as the restriction of $d$ onto $TU$ and $T\mathbb S^1$ respectively. We write $(x,p)=(x^1,\ldots,x^{2n-2},p_1,\ldots,p_{2n-2})\in T^*U$ and $(y,\mu)\in T^* \mathbb S^1$ within the proof of the theorem. Thus, the metric $d_{\mathcal F}$ has unit real $1\times 1$-matrix in the basis of the vertical vector fields $V$ and the corresponding co-metric $g_{\mathcal F}$ defined by the same matrix. Therefore, they are independent of the points $z\in \mathbb S^{2n-1}$. Moreover, since $V$ is normal and orthogonal to $\mathcal H$ everywhere we conclude that $H_{\mathcal F}(x,p,y,\mu)=\frac{1}{2}\mu^2$, where $\mu$ is the dual form to $V$ and counting the dimension of $\mu\in T^{\bot}$. The metric $d_{\mathcal H}$ is independent of $y$ by definition so does the co-metric $g_{\mathcal H}$. Moreover, the co-metric $g_{\mathcal H}$ is independent of $\mu$ by $\mu\in T^{\bot}$. We conclude, that $H_{\mathcal H}$ depends only on $(x,p)$. We calculate $$\{H_{\mathcal H},H_{\mathcal F}\}=\sum_{k=1}^{2n-2}\Big(\frac{\partial H_{\mathcal H}}{\partial p_k}\frac{\partial H_{\mathcal F}}{\partial x^k}-\frac{\partial H_{\mathcal H}}{\partial x^k}\frac{\partial H_{\mathcal F}}{\partial p_k}\Big)+\frac{\partial H_{\mathcal H}}{\partial \mu}\frac{\partial H_{\mathcal F}}{\partial y}-\frac{\partial H_{\mathcal H}}{\partial y}\frac{\partial H_{\mathcal F}}{\partial \mu}=0.$$
\end{proof}

\begin{theorem}
A normal sub-Riemanian geodesic $\gamma(s)$ starting at a point  $a\in \mathbb S^{2n-1}$ is given by the formula
\begin{equation}\label{sRgeod}
\gamma(s)=\gamma_R(s)e^{-is\,d( v,V(a))}=(a\cos(\|v\|s)+\frac{v}{\|v\|}\sin(\|v\|s))e^{-is\,d( v,V(a))},
\end{equation} where $v$ is the initial velocity of the Riemannian geodesic $\gamma_R(s)$: $\dot\gamma_R(a)=v$ at the initial point $a$, and $d(v,V(a))$ is the projection of the initial velocity to the vertical direction at $a$. 
\end{theorem}

\begin{proof}
We start with a discussion of relations between two definitions of  Riemannian geodesics: by the Riemannian exponential map and as the projection  of the solution to the Hamiltonian system to the underlying manifold. It is a well known connection, but we recall it for completeness.
To obtain a Riemannian geodesic $\gamma_R(s)$ starting at a point $a\in \mathbb S^{2n-1}$ with the initial velocity vector $v\in T_a\mathbb S^{2n-1}$ we can make use of the Riemannian exponential map
$\exp_d(sv)$ defined by the metric $d$.  Let us recall, that given a metric $d=\{d_{ij}\}$ the Riemannian Hamiltonian function $H_d(z,p)=\frac{1}{2}d^{ij}p_i(z)p_j(z)$ is defned, where the co-metric $d^{ij}$ is  the inverse matrix for the metric $d$. The solution to the Hamiltonian system defines the flow $\Phi(q_0,p_0,t)$ on the cotangent bundle $T^*\mathbb S^{2n-1}$, where $(q_0,p_0)$ is the initial point. Moreover, the metric tensor $d$ defines the identification of tangent and cotangent bundles. Thus, the Riemannian geodesic $\gamma_R(s)=\exp_d(sv)$ is the result of the following superpositions $$\exp_d(sv(a)): T_a\mathbb S^{2n-1}\ \ {\overset{\iota}\hookrightarrow}\ \ T\mathbb S^{2n-1}\ {\overset{d}\longleftrightarrow}\  T^*\mathbb S^{2n-1}\ {\overset {Pr}\longrightarrow} \ \mathbb S^{2n-1}.$$ Here the first map $\iota$ is the inclusion of the tangent space at $a\in\mathbb  S^{2n-1}$ into the tangent bundle $T^*\mathbb S^{2n-1}$. The second map is the identification of the tangent and cotangent bundles by means of the metric $d$, which also gives the initial data $(q_0,p_0)=(a, d(v,\cdot))$ for the flow $\Phi(q_0,p_0,s)$. The last map projects the Hamiltonian flow $\Phi(q_0,p_0,s)$ to the sphere. 

Pick up a vector $v\in T_a\mathbb S^{2n-1}$ and write $v=v_{\mathcal H}+v_{\mathcal F}$ due to the orthogonal decomposition~\eqref{orthdecom}. We can construct $\exp_{d_{\mathcal H}}(sv_{\mathcal H})$ and $\exp_{d_{\mathcal F}}(sv_{\mathcal F})$ by means of the procedure described above using the horizontal $H_{\mathcal H}$ and the vertical $H_{\mathcal F}$ Hamiltonian functions. The identification between the tangent and cotangent spaces is realized by the linear maps $d_{\mathcal H}\colon \mathcal HT\to T^*\mathbb S^{2n-1}$, $g_{\mathcal H}\colon T^*\mathbb S^{2n-1}\to T\mathbb S^{2n-1}$ and $d_{\mathcal F}\colon \mathcal FT\to T^*\mathbb S^{2n-1}$, $g_{\mathcal F}\colon T^*\mathbb S^{2n-1}\to T\mathbb S^{2n-1}$, respectively. Then $$\exp_{d}(tv)=\exp_{d_{\mathcal H}}(tv_{\mathcal H})\circ\exp_{d_{\mathcal F}}(tv_{\mathcal F})$$ since the Hamiltonian functions $H_{\mathcal H}$, $H_{\mathcal F}$, $H_d$ are all Poisson involutive, the corresponding flows  commute. We observe also that the group exponential map $\exp_{\mathbb S^1}\colon \mathfrak u(1)\to U(1)$ agrees with the Riemannian exponential map $\exp_{d_{\mathcal F}}\colon \mathfrak u(1)\cong\mathcal FT_a\to \mathbb S^1$~\cite{Milnor}. If we write $\gamma_R(s)=\exp_{d}(sv)$, $\gamma_{sR}(s)=\exp_{d_{\mathcal H}}(sv_{\mathcal H})$, $e^{is|v_{\mathcal F}|}=\exp_{d_{\mathcal F}}(sv_{\mathcal F})$, then \begin{equation}\label{eq1}\gamma_{sR}(t)=\gamma_R(s)e^{-is|v_{\mathcal F}|}=\gamma_R(s)e^{-is\, d(v,V(a))},\end{equation} here we used the notation $|v_{\mathcal F}|=d(v,V(a))$ for the projection of the initial velocity onto the vertical direction $V(a)$ at $a$.

It is well known that the Riemannian geodesic on sphere is a great circle and can be written as \begin{equation}\label{eq2}\gamma_R(s)=a\cos(\|v\|s)+\frac{v}{\|v\|}\sin(\|v\|s).\end{equation} Here $a$ is the starting point, $v$ is the initial velocity, and $\|v\|^2=d(v,v)$ is the real norm of~$v$. Combining~\eqref{eq1} and~\eqref{eq2} we obtain~\eqref{sRgeod}.

To be sure that the geodesic is horizontal we verify the product $$\langle \dot\gamma_{sR},\gamma_{sR}\rangle=-id(v,V(a))\langle\gamma_R,\gamma_R\rangle+\langle\dot\gamma_R,\gamma_R\rangle.$$ Since $\gamma_R\in \mathbb S^{2n-1}$, then $\langle\gamma_R,\gamma_R\rangle=1$, $\re\langle\dot\gamma_R,\gamma_R\rangle=0$, and we need to calculate $\im\langle\dot\gamma_R,\gamma_R\rangle$. We write $a_k=\alpha_k+i\beta_k$ and $v_k=\upsilon_k+i\omega_k$. Then $V(a)=\sum_{k=1}^{n}(-\beta_k\partial_{x_k}+ \alpha_k\partial _{y_k})$ and
\begin{equation*}\begin{array}{lll}
i\im\langle\dot\gamma_R(s),\gamma_R(s)\rangle & = &i \im\sum_{k=1}^{n}\Big(-(\alpha_k+i\beta_k)\|v\|\sin(\|v\|s)+(\upsilon_k+i\omega_k)\cos(\|v\|s)\Big)\\ 
 & \times &
 \Big((\alpha_k-i\beta_k)\cos(\|v\|s)+\frac{(\upsilon_k-i\omega_k)}{\|v\|}\sin(\|v\|s)\Big)\\
 & = & 
i \sum_{k=1}^{n}\big(-\beta_k\upsilon_k+\alpha_k\omega_k\big)=id(v,V(a)).
\end{array}\end{equation*}
This completes the proof of the theorem.
\end{proof}

\begin{remark}
We have used the ideas developed in~\cite{Montgomery} for general principal $G$-bundles, in our case for the circle bundle of $(2n-1)$-dimensional spheres and the sub-Riemannian metric given by the restriction of the Euclidean metric. The equations for geodesics in the particular case of $\mathbb S^3$ were obtained in~\cite{ChMV} by solving  the Hamiltonian system, in~\cite{BosRossi} by applying the general Lie group theory and in~\cite{HR} as a solution to a variational equation. It was shown in~\cite{GM1} that sub-Riemannian geometry on $\mathbb S^3$  induced by CR-stucture of $\mathbb S^3\hookrightarrow \mathbb C^2$, by the Hopf fibration, and by Lie group structure $\mathbb S^3$ of unit quaternions coincide. 
\end{remark}

\begin{remark}
The solution to a variational equation in~\cite{HR} depends on a real parameter $\lambda$ that the authors call curvature. This parameter for a geodesic starting at the point $a$ coincides with $d(v,V(a))$ as was shown in~\cite{GM2}. The variational equation can not be generalized into the higher dimensions, but as we will see the value $d(v,V(a))$ will play an important role in the behavior of geodesics. We continue to use the notation $|v_{\mathcal F}(a)|=d(v,V(a))$.
\end{remark}

\begin{remark}\label{actionSU}
The group $\SU(n)$ acts in $\mathbb C^n$  transitively  on $\mathbb S^{2n-1}=\SU(n)/\SU(n-1)$, moreover, the action is conformal isometrically. Therefore, the normal vector field $N$ is invariant under the action of  $\SU(n)$. Since the orthogonal structure $\mathcal HT\oplus \spn_{\mathbb C}(N+iN)$ is preserved under the action of $\SU(n)$ we conclude that the vector field $V=iN$ is also invariant under the action of $\SU(n)$. Thus if $\phi\in \SU(n)$, then $$\mathbb S^{2n-1}\ni a\mapsto \phi a\in\mathbb S^{2n-1},$$
$$\mathcal FT_a\ni V(a)\mapsto \phi_{*} V(a)=V(\phi a)\in \mathcal FT_{\phi a},\quad \mathcal HT_a\mapsto \mathcal HT_{\phi a}.$$ Furthermore,
$$|v_{\mathcal F}(a)|=d(v,V(a))=d(\phi v,V(\phi a))=|v_{\mathcal F}(\phi a)|,\ \quad d(v,v)=d(\phi v,\phi v)$$ $$v_{\mathcal H}(a)\mapsto \phi v_{\mathcal H}(a)=v_{\mathcal H}(\phi a).$$

Any geodesic parametrized by arc length, that is $\|v_{\mathcal H}\|=1$, is uniquely determined by the initial point $a$,  and the initial horizontal $v_{\mathcal H}(a)$ and the vertical $v_{\mathcal F}(a)$ velocities. Let us denote such a curve by $\gamma(s;a,v_{\mathcal H},v_{\mathcal F})$. Then $$\phi\gamma(s;a,v_{\mathcal H},v_{\mathcal F})=\gamma(s;\phi a,\phi_* v_{\mathcal H}, \phi_*v_{\mathcal F}).$$
\end{remark}

\section{Open and closed sub-Riemannian geodesics on $\mathbb S^{2n-1}$}

Let us suppose that geodesic~\eqref{sRgeod} is parametrized by  arc length: $\|v_{\mathcal H}\|=1$. Then $\|v\|=\sqrt{1+|v_{\mathcal F}|^2}$. We are looking for the points of intersection with the vertical fiber at $a$, whose equation is given by $ae^{i\phi}$, $\phi\in[0,2\pi)$. 

\begin{theorem}\label{opencl}
Under the above  notations we claim
\smallskip

\noindent {\bf 1.} If $\frac{|v_{\mathcal F}|}{\sqrt{1+|v_{\mathcal F}|^2}}=\frac{p}{q}<1$ is rational with $p$, $q>0$ relatively prime, then for $n=(2q)m$, $m=1,2,\ldots$, the geodesic is a closed curve which meets the base point $a$ for the first time for $m=1$, and then, it is periodic with  minimal period $2\pi\sqrt{q^2-p^2}$. The length of each closed loop of the geodesic is $2\pi\sqrt{q^2-p^2}$. The geodesic intersects the fiber at points $ae^{i\pi n\Big(\frac{q-p}{q}\Big)}$, $n=1,2,\ldots, 2q$, for the first time and then periodically returns back. The length of the part of the geodesic between the intersection points is equal to $\pi\sqrt{\frac{q^2-p^2}{q^2}}$.
\smallskip

\noindent{\bf 2.} If $\frac{|v_{\mathcal F}|}{\sqrt{1+|v_{\mathcal F}|^2}}$ is irrational, then the geodesic is open and is diffeomorphic to the straight line. It meets the fiber for $\hat s=\frac{\pi}{\sqrt{1+|v_{\mathcal F}|^2}}n$, $n=1,2,\ldots$, at the points $ae^{i\pi n\Big(1-\frac{|v_{\mathcal F}|}{\sqrt{1+|v_{\mathcal F}|^2}}\Big)}$.
\end{theorem}

\begin{proof}
In order to proof the theorem we need to find the moment $s=\hat s$, such that the equation $$e^{-is|v_{\mathcal F}|}\Big(a\cos\big(s\sqrt{1+|v_{\mathcal F}|^2}\big)+\frac{v}{\sqrt{1+|v_{\mathcal F}|^2}}\sin\big(s\sqrt{1+|v_{\mathcal F}|^2}\big)\Big)=a$$ is verified. Due to Remark~\ref{actionSU} it is sufficient to check the base point $a=(1,0,\dots,0)$.
 It follows that for $s=\hat s=\frac{\pi}{\sqrt{1+|v_{\mathcal F}|^2}}n$, $n=1,2,\ldots$, the geodesic intersects the fiber at the points $ae^{i\pi n\Big(1-\frac{|v_{\mathcal F}|}{\sqrt{1+|v_{\mathcal F}|^2}}\Big)}$.

If the geodesic is closed, then $ae^{i\pi n\Big(1-\frac{|v_{\mathcal F}|}{\sqrt{1+|v_{\mathcal F}|^2}}\Big)}=a$. We conclude that if $\frac{|v_{\mathcal F}|}{\sqrt{1+|v_{\mathcal F}|^2}}$ is rational and $\frac{|v_{\mathcal F}|}{\sqrt{1+|v_{\mathcal F}|^2}}=\frac{p}{q}<1$ where $p,q$ are relatively prime and positive, then for $n=2q$ the geodesic  returns back to $a$ for the fist time, and then it becomes periodic passing through the point $a$ for each $n=(2q)m$, $m=1,2,3,\ldots$. The minimal period is obtained by setting $n=2q$ and $|v_{\mathcal F}|^2=\frac{p^2}{q^2-p^2}$ in $\hat s$ and equals $2\pi\sqrt{q^2-p^2}$. The Carnot-Carath\'eodory length of one closed loop is equal to $\hat s=2\pi\sqrt{q^2-p^2}$. The first loop of a closed geodesic at the moments $\hat s=\frac{\pi\sqrt{q^2-p^2}}{q}n$, $n=1,2,\ldots,2q$,  intersects the vertical circle at the points $ae^{i\pi n\Big(\frac{q-p}{q}\Big)}$. The Carnot-Carath\'eodory length of each part between intersections is~$\frac{\pi\sqrt{q^2-p^2}}{q}$.

If the fraction $\frac{|v_{\mathcal F}|}{\sqrt{1+|v_{\mathcal F}|^2}}$ is irrational, then the geodesic is open and is diffeomorphic to the straight line. It also intersects the vertical circle at the moments $s=\frac{\pi}{\sqrt{1+|v_{\mathcal F}|^2}}n$ at the points $ae^{i\pi n\Big(1-\frac{|v_{\mathcal F}|}{\sqrt{1+|v_{\mathcal F}|^2}}\Big)}$. The Carnot-Carath\'eodory length of each part between intersections is $\frac{\pi}{\sqrt{1+|v_{\mathcal F}|^2}}$.
\end{proof}

We can prove the following theorem (see also~\cite{HR}) using the group structure of $\mathbb S^3$, which is isomorphic to $\SU(2)$. The elements of $\SU(2)$ are the matrices
\[
\left(
\begin{array}{cc}
 z_1 & z_2   \\
- \bar{z}_2 & \bar{z}_1 \end{array}
\right),
\]
where $z_1,z_2\in\mathbb C$, and $|z_1|^2+|z_2|^2=1$, hence we also denote the elements of $\SU(2)$ by unit vectors $(z_1,z_2)$. Denote by $e$ the unity  $(1,0)$ of the group.

\begin{theorem}
The left translation $\phi\gamma$ of a geodesic $\gamma(s;e,v_{\mathcal H},|v_{\mathcal F}|)$ by $\phi\in \SU(2)$, $\phi=(\rho,iv\sqrt{1-\rho^2})$, where $\rho^2=\big(1+(\sqrt{1+|v_{\mathcal F}|^2}-|v_{\mathcal F}|)^2\big)^{-1}$, belongs to the Clifford torus $\{(w_1,w_2)\in S^3: |w_1|^2=\rho^2\}$.
\end{theorem}

\begin{proof}
To prove the theorem we need to find an element $\phi=(\phi_1,\phi_2)\in \SU(2)$, such that $$\phi\gamma(s;e,v_{\mathcal H},|v_{\mathcal F}|)=(w_1(s),w_2(s)),\qquad\text{with}\quad |w_1|^2=\rho^2, $$ where $\rho$ is a constant depending on $|v_{\mathcal F}|$. We remind that the action of $\SU(2)$ on the sphere $\mathbb S^3$ is defined by $$\phi(z_1,z_2)=(\phi_1,\phi_2)(z_1,z_2)=(\phi_1z_1-\phi_2\bar z_2,\phi_1z_2+\phi_2\bar z_1).$$

The initial velocity vector $v$ at the point $(1,0)$ has the form $v=(i|v_{\mathcal F}|,e^{i\alpha})$, where we write $e^{i\alpha}$ for the initial horizontal velocity $v_{\mathcal H}$. According to Remark~\ref{actionSU}, the coordinate $w_1$ can be written as $$w_1=e^{-is|v_{\mathcal F}|}\Big(\phi_1\cos\big(s\sqrt{1+|v_{\mathcal F}|^2}\big)+\frac{(\phi v)_1}{\sqrt{1+|v_{\mathcal F}|^2}}\sin\big(s\sqrt{1+|v_{\mathcal F}|^2}\big)\Big),$$ where $(\phi v)_1=i|v_{\mathcal F}| \phi_1-e^{-i\alpha}\phi_2$. Since we are interested only in a fixed value of $|w_1|=\rho>0$, we can choose two free parameters in $\phi=(\phi_1,\phi_2)$. We set $\phi_1=\rho$. Then $|w_2|=\sqrt{1-\rho^2}$, and we take $\arg\phi_2=\alpha-\frac{\pi}{2}$ in order to make $(\phi v)_1$ pure imaginary and to simplify the calculation of $|w_1|^2$. Then $$|w_1|^2=\rho^2\Big(\cos^2\big(s\sqrt{1+|v_{\mathcal F}|^2}\big)+\frac{(|v_{\mathcal F}|+\frac{\sqrt{1-\rho^2}}{\rho})^2}{1+|v_{\mathcal F}|^2}\sin^2\big(s\sqrt{1+|v_{\mathcal F}|^2}\big)\Big).$$ The equality $\frac{(|v_{\mathcal F}|+\frac{\sqrt{1-\rho^2}}{\rho})^2}{1+|v_{\mathcal F}|^2}=1$ gives $\rho^2=\big(1+(\sqrt{1+|v_{\mathcal F}|^2}-|v_{\mathcal F}|)^2\big)^{-1}$. This proves the theorem.
\end{proof}

\section{Boundary value problem and distance on $\mathbb S^{2n-1}$}

Let us find the distance from the point $a$ to a point $p$ in a fiber, i.~e., $p\in a e^{i\omega}$, $\omega\in(0,\pi)$. 

\begin{theorem}\label{th4}
The Carnot-Carath\'eodory distance $d_{c-c}(a,p)$ from $a\in \mathbb S^{2n-1}$ to the point $p=ae^{i\omega}$ is $$d_{c-c}(a,p)=\sqrt{\omega(2\pi-\omega)}.$$
\end{theorem}

\begin{proof}
We assume that the geodesic $\gamma(s,a,v_{\mathcal H},|v_{\mathcal F}|)$ is parametrized by arc length and $\omega\in(0,\pi)$. We need to solve the equation $$e^{-is|v_{\mathcal F}|}\Big(a\cos\big(s\sqrt{1+|v_{\mathcal F}|^2}\big)+\frac{v}{\sqrt{1+|v_{\mathcal F}|^2}}\sin\big(s\sqrt{1+|v_{\mathcal F}|^2}\big)\Big)=ae^{i\omega}.$$
Arguing as in Theorem~\ref{opencl}, we conclude that the geodesic intersects the fiber at the moments $\hat s=\frac{\pi}{\sqrt{1+|v_{\mathcal F}|^2}}n$ at the points $ae^{i\pi n\Big(1-\frac{|v_{\mathcal F}|}{\sqrt{1+|v_{\mathcal F}|^2}}\Big)}$, $n=1,2,\ldots$. We are interested in finding $|v_{\mathcal F}|$, such that the geodesic $\gamma(s,a,v_{\mathcal H},|v_{\mathcal F}|)$ meets the point $ae^{i\omega}$ for the first time, or in other words for $n=1$. Thus, $$\pi\Big(1-\frac{|v_{\mathcal F}|}{\sqrt{1+|v_{\mathcal F}|^2}}\Big)=\omega\ \ \Longrightarrow\ \ |v_{\mathcal F}|^2=\frac{(\pi-\omega)^2}{\omega(2\pi-\omega)}\ \ \Longrightarrow\ \ \hat s=\sqrt{\omega(2\pi-\omega)}.$$ Since the geodesic is parametrized by  arc length, the value $\hat s$ gives us the length of the geodesic joining the points $a$ and  $p$. If the point $p$ tends to $a$ (or in other words $\omega\to 0$), then the velocity $|v_{\mathcal F}|$ tends to infinity and the length tends to 0.

The geodesic is not unique. Varying the directions of the horizontal velocities we obtain uncountably many geodesics parametrized by the $(2n-3)$-sphere. But all of them have the same length.

There are geodesics that end at point $p$  at times $\hat s=\frac{\pi}{\sqrt{1+|v_{\mathcal F}|^2}}n$, $n>1$. The initial velocity for the $n$-th case is   $$(|v_{\mathcal F}|)_n=\frac{\pi n-\omega}{\omega(2\pi n-\omega)}.$$ For any fixed $\omega$ we have $$(|v_{\mathcal F}|)_n\ \ {\underset{n\to\infty}\longrightarrow}\ \ \frac{1}{2\omega},\qquad \hat s_n=\sqrt{\omega(2\pi n-\omega)}\ \ {\underset{n\to\infty}\longrightarrow}\ \ \infty.$$ 

If $\omega \in (\pi,2\pi)$, then we can switch $\omega$ to $-\omega$ by  spherical symmetry. 
If $\omega\in(0,2\pi)+2\pi m$, $m=1,2,\ldots$, then in order  to find $|v_{\mathcal F}|$ we have to solve the equation $\Big(1-\frac{|v_{\mathcal F}|}{\sqrt{1+|v_{\mathcal F}|^2}}\Big)=\frac{\omega+\pi m}{\pi n}$ for different combinations of $m$ and $n$. The condition $\frac{\omega+\pi m}{\pi n}\in(0,1)$ reduces this case to the one considered above with the argument $\omega\in(0,\pi)$. 

Since $\sqrt{\omega(2\pi-\omega)}$ is the minimal length among all the geodesics,  it gives the Carnot-Carath\'eodory distance, see~\cite{Strichartz}. 
\end{proof}

\subsection{Boundary value problem and distance on $\mathbb S^3$}

Since the study of boundary value probem for arbitrary points of $\mathbb S^{2n-1}$, $n>2$ is rather difficult, we concentrate our attention on
the case of $\mathbb S^3$. As it was shown in \cite{GM1} the sub-Riemannian structure given by the Hopf fibration and by the group structure $\SU(2)$ on $\mathbb S^3$ coincide. Therefore, we simplify considerations taking the base point $a=1=(1,0)$ in complex coordinates. If $v=(v_0+iv_1,v_2+iv_3)$, then $v_0=0$, $v_1=|v_{\mathcal F}|$, and
$v_2$, $v_3$ are arbitrary.
The formulas for geodesics  turn into $(z_1(s),z_2(s))$, where
\[
z_1(s)=e^{-iv_1s}(\cos (\|v\|s)+i\frac{v_1}{\|v\|}\sin(\|v\|s)),
\]
\[
z_2(s)=e^{-iv_1s}\frac{v_2+iv_3}{\|v\|}\sin(\|v\|s).
\]
In what follows  we do not assume parametrization of geodesics by arc length, we suppose that all of them are parametrized in the interval $[0,1]$.
For convenience, we rewrite geodesic equations with the following notations
\[
u:=\frac{v_1}{\|v\|}, \quad  re^{i\alpha}:=\frac{v_2+iv_3}{\|v\|},\quad  \rho:=\|v\|.
\]
Then
\begin{equation}\label{z1s}
z_1(s)=e^{-i u\rho  s}(\cos \rho s+iu\sin \rho s),
\end{equation}
\begin{equation}\label{z2s}
z_2(s)=re^{-i(u \rho  s+\alpha)}\sin \rho s.
\end{equation}

Let us denote the endpoint at $s=1$ of the geodesic by $(z_1,z_2)$, $z_1=|z_1|e^{i\theta_1}$, $z_2=|z_2|e^{i\theta_2}$. Then,
\begin{equation}\label{z1}
z_1=e^{-i u\rho }(\cos \rho +iu\sin \rho),
\end{equation}
\begin{equation}\label{z2}
z_2=re^{-i(u \rho +\alpha)}\sin \rho .
\end{equation}
Given an endpoint $(z_1,z_2)$, i.e.,  given the values of $|z_1|$, $\arg z_1\in [-\pi,\pi)$, and  $\arg z_2\in [-\pi,\pi)$,  the unknown  parameters are $u$, $\rho$, and $\alpha$. The last parameter $\alpha$ is the simplest one, which one defines it at the end of all computations relating it to $\arg z_2$.

\begin{remark}\label{rem4}
Observe that $u$ may belong only to the open interval $u\in(-1,1)$. If $u=\pm 1$, then $r=0$ by $u^2+r^2=1$, and $z_2(s)\equiv 0$, which implies that  formulas (\ref{z1s}--\ref{z2s}) reduce to the fixed point $(1,0)$. 
\end{remark}

\begin{itemize}
\item {\bf Exceptional cases.} Let us start with the cases when the  endpoint $(z_1,z_2)$ lies on the vertical line or on the horizontal sphere $\mathbb S^2$.
In the first case $z_2=0$. 

\begin{itemize}
\item {\bf Vertical line and loops.}  The general case was considered in the previous subsection in Theorem~\ref{th4}. 
\item {\bf Antipodal point} $(-1,0)$ is the intersection of the vertical line starting at $(1,0)$ and $\mathbb S^2$ considered as the base space for the fiber at $(1,0)$. We distinguish this case because this point is the intersection of the horizontal sphere $\mathbb S^2$ and the vertical line, i.e., these points can be connected with geodesics either lying on $\mathbb S^2$ for all $s\in [0,1]$, or leaving $\mathbb S^2$.

\begin{proposition}
The geodesics connecting the antipodal points $(1,0)$ and $(-1,0)$ are given by formulas \eqref{z1s} and \eqref{z2s}, $s\in[0,1]$ with
\begin{itemize}
\item $\rho=\pi m$, $m\in\mathbb N$;
\item $u=(2p+1)/m$ for even $m$, where $p$ is integer and $-\frac{m}{2}\leq p\leq \frac{m}{2}-1$;
\item $u=2p/m$ for odd $m$,  where $p$ is integer and $-\frac{m-1}{2}\leq p\leq \frac{m-1}{2}$;
\item $\alpha$ is arbitrary.
\end{itemize}
The length of these geodesics is $d=\pi\sqrt{m^2-(2p+1)^2}$ for even $m$ and $d=\pi\sqrt{m^2-4p^2}$ for odd $m$.
The Carnot-Carath\'eodory distance is $d_{c-c}= \pi$ and it is realized by the geodesic with $u=0$, $\rho=\pi$, which lies on $\mathbb S^2$.
\end{proposition}
\begin{proof}
   The value $r\neq 0$,  see Remark~\ref{rem4}. So we have $\rho=\pi m> 0$. We avoid the case $m=0$ because  the speed is 0 and there is no motion at all. Then $e^{-i\pi m u}\cos\pi m=-1$ and
\[
u=\frac{2p+1}{2q}, \quad q\in \mathbb N \quad \mbox{or\  \ } u=\frac{2p}{2q+1},\quad q\in \mathbb N\cup\{0\}, \quad p\in \mathbb Z,
\]
where $m=2q$ or $2q+1$ respectively. In the first case, $-q \leq p\leq q-1$, and in the second $-q \leq p\leq q$. The rest of the unknowns are $r=\sqrt{1-u^2}$ and an arbitrary $\alpha$. The length of the geodesics is defined as
\[
d_{2p+1,2q}=\pi\sqrt{4q^2-(2p+1)^2},  \quad \mbox{or\  \ } d_{2p,2q+1}=\pi\sqrt{(2q+1)^2-4p^2}.
\]
In the first case, the minimal length of geodesics $d=\pi\sqrt{3}$ is realized for $q=1$, $p=0$ or $q=1$, $p=-1$. In the second case,  $d=\pi$ for $q=p=0$.
Thus, the Carnot-Carath\'eodory distance is given for $u=0$ and the corresponding geodesic joining $(1,0)$ and $(-1,0)$ lies on $\mathbb S^2$  and it is the half of a big circle (mod$(\alpha)$). 
\end{proof}

{\bf Example}. Let us write down explicit formulas for three geometrically different  (with respect to the rotational symmetry in $\alpha$)  geodesics for $u=-1/2, 0, 1/2$. For $u=\pm 1/2$, the geodesics do not lie on $\mathbb S^2$ whereas for $u=0$ the geodesic lies entirely on $\mathbb S^2$.
\[
\left\{\begin{array}{lll}
z_1(s)&=&e^{\pm i\pi   s}(\cos 2\pi s\pm\frac{i}{2}u\sin 2\pi s),\\

z_2(s)&=&\frac{\sqrt{3}}{2}e^{i(\pm \pi s+\alpha)}\sin 2\pi  s,
\end{array}\right.
\left\{\begin{array}{lll}
z_1(s)&=&\cos \pi s,\\

z_2(s)&=&e^{i\alpha}\sin \pi  s.
\end{array}\right.
\]
The last geodesic is the minimizer giving the  Carnot-Carath\'eodory distance.

\item {\bf Horizontal sphere.} If the ending point $(z_1,z_2)$ belongs to the vertical line, then there are infinitely many geodesics joining the origin with $(z_1,z_2)$. This phenomenon is
typical for sub-Riemannian geometry and it is seen in many examples, for instance, for the Heisenberg group. However, in the case of the sub-Riemannian sphere, the number
of geodesics joining two points even of the horizontal sphere is also infinite.

\begin{proposition}
Let $(z_1,z_2)\in \mathbb S^2$. There are countable number of geodesics connecting the points $(1,0)$ and $(z_1,z_2)$. The value $z_1$ is real. These geodesics are given by formulas \eqref{z1s} and \eqref{z2s}, $s\in[0,1]$ where
\begin{itemize}
\item $\rho=\rho_m$, $m\in\mathbb N$ are solutions to the equation
\[
\frac{\cos\rho}{{\displaystyle\cos \left(\rho\,\frac{\sqrt{z_1^2-\cos^2\rho}}{|\sin\rho|}\right)}}= z_1;
\]
\item $u=u_m=\pm\sqrt{\frac{z_1^2-\cos^2\rho_m}{\sin^2\rho_m}};$
\item $\alpha=\arg z_2-u_m\rho_m$.
\end{itemize}
The Carnot-Carath\'eodory distance is $d_{c-c}= |\arccos z_1|$ and it isrealized by the geodesic with $u=0$, $\rho=\arccos z_1$, which lies on $\mathbb S^2$.
\end{proposition}
\begin{proof}
 If the point $(z_1,z_2)$ belongs to the sphere $\mathbb S^2$, then $\im z_1=0$. Let us assume $z_1\in (0,1)$. This implies
\begin{equation}\label{mnim}
u\sin \rho \cos u\rho=\cos \rho\sin u\rho.
\end{equation}

If $u=0$, then there is a unique geodesic lying on $\mathbb S^2$ joining $(1,0)$ with $(z_1,z_2)$ modulo repeating big circles, $\im z_2=0$. Its minimal length is $d=|\arccos z_1|$. Indeed, we have
$\|v\|=\rho$ in this case which presents the horizontal speed and which is preserved under the motion.

If $\cos \rho\sin u\rho=0$ and $u\neq 0$, then $\rho=\frac{\pi}{2}+\pi n$ or $u\rho=\pi n$. In both cases we come to the conclusion that $u=\pm 1$, which we get rid of, see Remark~\ref{rem4}.

In what follows we consider only the case $u\in [0,1)$ because the case $u\in (-1,0]$ is treated similarly. 
If $\cos \rho\sin u\rho$ does not vanish, then all values $\sin \rho$, $\cos \rho$, $\sin u\rho$, $\cos u\rho$ are non-vanishing too, and hence, $u\rho, \rho\neq \pi n$ and $u\rho, \rho\neq \frac{\pi}{2}+\pi n$. So the parameters $u$ and $\rho$ satisfy the system of equations
\begin{equation}\label{system}
\frac{\tan u\rho}{u\rho}=\frac{\tan\rho}{\rho},\quad z_1=\frac{\cos \rho}{\cos u\rho}.
\end{equation}
Combining them
we obtain the explicit function
\begin{equation}\label{u}
u^2=\frac{z_1^2-\cos^2\rho}{\sin^2\rho},
\end{equation}
which is defined in each interval $\rho\in D_n\equiv (\arccos z_1+\pi n,\pi(n+1)-\arccos z_1)$, which is open because $u\neq 0$ and we choose $\arccos z_1\in (0,\pi/2)$.
Substituting $u=u(\rho)$ in any of two equations from the system \eqref{system} we obtain the equation
\begin{equation}\label{equation}
\frac{\cos\rho}{{\displaystyle\cos \left(\rho\,\frac{\sqrt{z_1^2-\cos^2\rho}}{|\sin\rho|}\right)}}= z_1.
\end{equation}
Let us denote the left-hand side of the latter equality by $\Phi(\rho)$ for every $z_1\in (-1,1)$ fixed, see its graph in Figure~1.
The function $\Phi$ is rather complicated to investigate completely. 
\vspace{0pt}
\begin{figure}\label{Graph1}
\begin{center}
\begin{pspicture}(0,0)(7,7)
\put(0,0){\epsfig{file=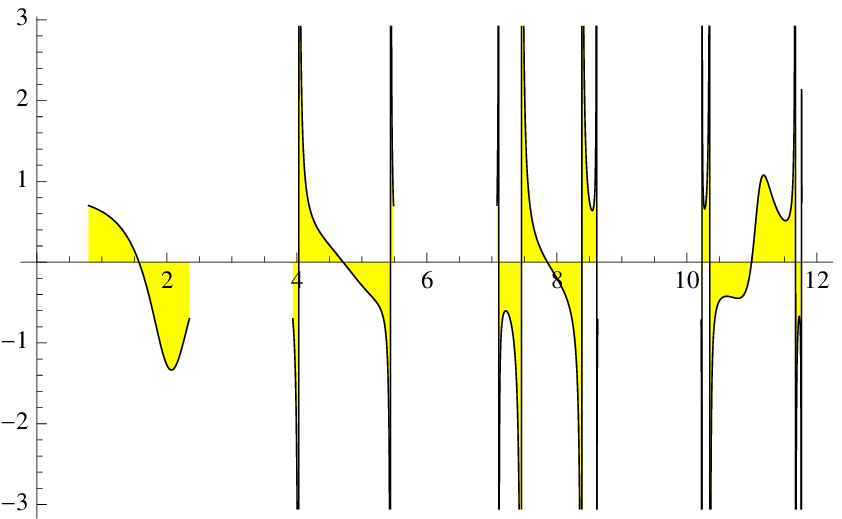,height=7cm}}
\rput(1.9,3.95){$ \frac{\pi}{2}$}
\rput(3.2,3.85){$\pi$}
\rput(4.6,3.95){$ \frac{3\pi}{2}$}
\rput(7.5,3.95){$ \frac{5\pi}{2}$}
\rput(6,3.85){$2\pi$}
\rput(8.7,3.85){$3\pi$}
\rput(11.7,3.49){$\rho$}
\rput(0.15,4.28){$z_1$}
\rput(1.8,3){$D_0$}
\rput(4.6,3){$D_1$}
\rput(7.4,3){$D_2$}
\rput(10.5,3){$D_3$}
\rput(2,6){$\Phi(\rho)$}
\pscircle[fillstyle=solid,
fillcolor=black](3.25,3.49){.05}
\pscircle[fillstyle=solid,
fillcolor=black](7.45,3.49){.05}
\pscircle[fillstyle=solid,
fillcolor=black](1.91,3.49){.05}
\pscircle[fillstyle=solid,
fillcolor=black](4.67,3.49){.05}
\pscircle[fillstyle=solid,
fillcolor=black](6,3.49){.05}
\pscircle[fillstyle=solid,
fillcolor=black](0.5,4.27){.05}
\end{pspicture}
\end{center}
\caption[]{Graph of the function $\Phi(\rho)$ for $z_1=0.7$.}
\end{figure}

However, the derivative is calculated as
\[
\Phi'(\rho)=-\frac{\sin\rho}{\cos u\rho}(1-u^2)(1-\rho\cot\rho),
\]
substituting $u$ from \eqref{u}. Observe again that the  function $\Phi$ is defined only on $\bigcup_{n=1}^{\infty}D_n$ where it vanishes at the points $\frac{\pi}{2}+\pi n\in D_n$. We have
\[
\Phi'(\frac{\pi}{2}+\pi n)=\frac{(-1)^{n+1}(1-z_1^2)}{\cos(z_1(\frac{\pi}{2}+\pi n))}\neq 0,\infty.
\]
The first equation from system \eqref{system} implies that solution to (\ref{equation}) may be searched only for $\rho>\pi$.
The vertical asymptotes are at the points $\rho=\rho_m$, where $\rho_m$ are the roots of the equation
\begin{equation}\label{Psi}
\rho\,\frac{\sqrt{z_1^2-\cos^2\rho}}{|\sin\rho|}=\frac{\pi}{2}+\pi m, \quad m=1,2,3\dots \,\,.
\end{equation}
Let us denote by $\Psi(\rho)$ the left-hand side of the latter equation. The function $\Psi(\rho)$ is continuous in each interval $D_n$ and vanishes at its endpoints, 
see Figure~2 and ~3.
Moreover, 
\[
z_1(\frac{\pi}{2}+\pi n)<\max\limits_{D_n}\Psi(\rho)<(n+1)\pi-\arccos z_1
\]
in this interval.
\vspace{0pt}
\begin{figure}\label{Graph12}
\begin{center}
\begin{pspicture}(0,0)(7,7)
\put(0,0){\epsfig{file=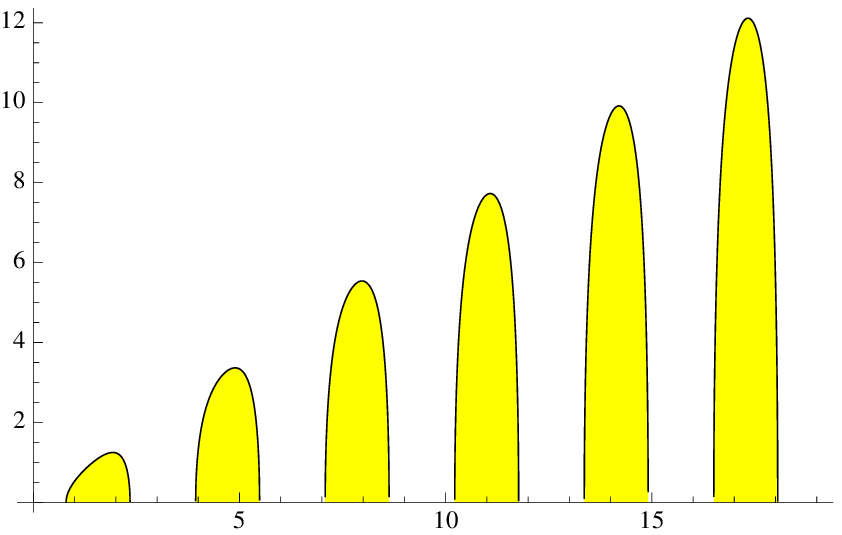,height=7cm}}
\rput(11.2,0.4){$\rho$}
\rput(1.3,0){$D_0$}
\rput(2.8,0){$D_1$}
\rput(4.7,0){$D_2$}
\rput(6.5,0){$D_3$}
\rput(8.1,0){$D_4$}
\rput(9.8,0){$D_5$}
\rput(4,4){$\Psi(\rho)$}
\end{pspicture}
\end{center}
\caption[]{Graph of the function $\Psi(\rho)$ for $z_1=0.7$.}
\end{figure}

\vspace{0pt}
\begin{figure}\label{Graph2}
\begin{center}
\begin{pspicture}(0,0)(7,7)
\put(0,0){\epsfig{file=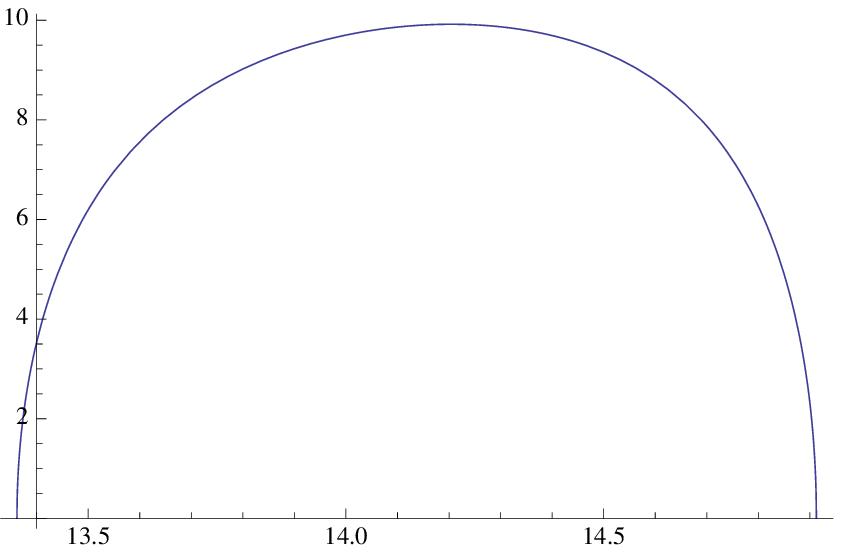,height=6cm}}
\rput(0,1.15){$\frac{\pi}{2}$}
\rput(0,3){$\frac{3\pi}{2}$}
\rput(0,4.4){$\frac{5\pi}{2}$}
\rput(-1.3,0.7){$\arccos z_1+4\pi$}
\rput(-1.1,5.7){$z_1(\frac{\pi}{2}+4\pi)$}
\rput(7.7,0.7){$5\pi-\arccos z_1$}
\rput(4.5,0.7){$\frac{9\pi}{2}$}
\rput(9.5,0.3){$\rho$}
\psline[linestyle=dashed, linewidth=0.015](0.5,4.6)(9,4.6)
\psline[linestyle=dashed, linewidth=0.015](0.5,3)(9,3)
\psline[linestyle=dashed, linewidth=0.015](0.5,1.2)(9,1.2)
\psline[linestyle=dashed, linewidth=0.03](0.5,5.7)(9,5.7)
\pscircle[fillstyle=solid,
fillcolor=black](0.2,0.35){.05}
\pscircle[fillstyle=solid,
fillcolor=black](9,0.35){.05}
\pscircle[fillstyle=solid,
fillcolor=black](4.5,0.35){.05}
\end{pspicture}
\end{center}
\caption[]{Zoom of the graph of the function $\Psi(\rho)$ for $z_1=0.7$ in the interval $D_4$.}
\end{figure}
Thus, there are at least two values of $m$ for $n\in \mathbb N$, $n>n_0\equiv n_0(z_1)$, such that solutions to equation~\eqref{Psi} exist and they are different from $\frac{\pi}{2}+\pi n$. Let us denote by $\rho'_n$ and
$\rho''_n$ the solutions to~\eqref{Psi}, such that $\frac{\pi}{2}+\pi n\in (\rho'_n,\rho''_n)$ and there are no other solutions in this interval.
Then the function $\Phi(\rho)$ has asymptotes at $\rho'_n$ and
$\rho''_n$, it is continuous in the interval $(\rho'_n,\rho''_n)$, vanishes at $\frac{\pi}{2}+\pi n$, and ranges in $(-\infty, \infty)$ on $ (\rho'_n,\rho''_n)$.
Therefore, the equation~\eqref{equation} has at least one solution in $(\rho'_n,\rho''_n)$. Since $n$ ranges in $\mathbb N$, $n>n_0$ for $n_0$ sufficiently large, we have an infinite number
of geodesics joining two points on the horizontal sphere $\mathbb S^2$.

 The inequality $\frac{\pi}{2}\sqrt{1-x^2}>\arccos x$ for $x\in(0,1)$ assures that the geodesic entirely lying on $\mathbb S^2$ realizes the minimal distance.
 \end{proof}

\end{itemize}
\item {\bf General case.} Assume that the endpoint $(z_1,z_2)$ does belong neither to the vertical line nor to the horizontal sphere $\mathbb S^2$.
The equation \eqref{z2} implies 
\[
r\rho\geq r|\sin \rho|=|z_2|,
\]
where $r\rho$ is the length of the geodesic. At the same time, the relation
\[
\cot \rho=\sigma_1\sigma_2\frac{\sqrt{|z_1|^2-u^2}}{|z_2|}= \sigma_1\sigma_2\frac{\sqrt{r^2-|z_2|^2}}{|z_2|},
\]
follows from  \eqref{z2} with
 \[
\sigma_1=\sgn[\cos \rho],\quad \sigma_2=\sgn[\sin \rho].
\]
Observe that $\rho=\pi m$ corresponds to the exceptional case $(z_1,z_2)$ in the vertical line, and hence, the function $\cot \rho$ is well-defined and finite.
The equation  \eqref{z1} implies
\[
e^{i(\theta_1+\rho u)}=\frac{\cos \rho+iu\sin\rho}{\sqrt{\cos^2\rho+u^2\sin^2\rho}}=\frac{\sigma_1\sqrt{|z_1|^2-u^2}+i\sigma_2 u|z_2|}{r|z_1|},\quad \rho\neq \pi m.
\]
\vspace{0pt}
\begin{figure}\label{ffig1}
\begin{center}
\begin{pspicture}(0,0)(7,7)
\put(0,0){\epsfig{file=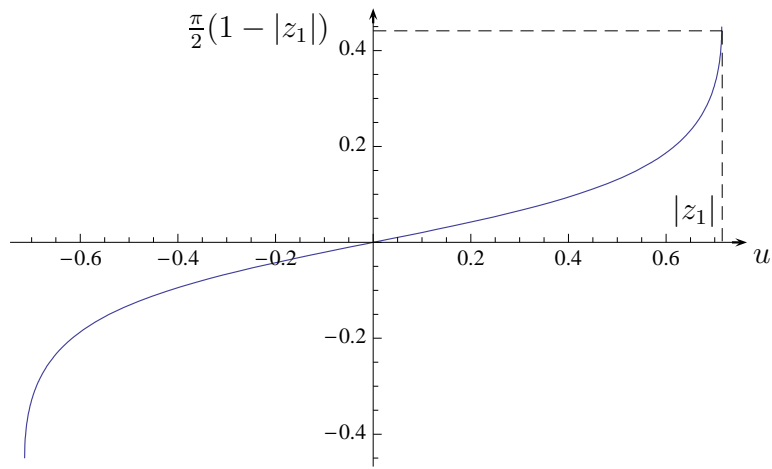,height=6cm}}
\psline[linecolor=black, linewidth=0.3mm]{->}(9.6,2.985)(9.8,2.985)
\psline[linecolor=black, linewidth=0.3mm]{->}(4.83,5.9)(4.83,6.1)
\rput(10,2.8){$u$}
\psline[linecolor=black, linestyle=dashed, linewidth=0.1mm]{-}(9.47,5.8)(9.47,3)
\rput(9.1,3.4){$|z_1|$}
\psline[linecolor=black, linestyle=dashed, linewidth=0.1mm]{-}(4.83,5.8)(9.47,5.8)
\rput(3.3,5.8){$\frac{\pi}{2}(1-|z_1|)$}
\end{pspicture}
\end{center}
\caption[]{Graph of the function $B(u)$.}
\end{figure}

\begin{itemize}
\item {\bf Case $\sigma_1>0$, $\sigma_2>0$.} The equations for $u$ and $\rho>0$ become
\[
\rho=\arccot\frac{\sqrt{|z_1|^2-u^2}}{|z_2|}+2\pi q,\quad q\in \mathbb N\cup\{0\},
\]
\[
\theta_1=\arccot\frac{\sqrt{|z_1|^2-u^2}}{u|z_2|}-u\arccot\frac{\sqrt{|z_1|^2-u^2}}{|z_2|}+2\pi(p-uq),\quad p\in \mathbb Z.
\]
where the branch of $\arccot$ is chosen to be in the interval $(0,\pi)$. Observe that 
\[
\arccot\frac{\sqrt{|z_1|^2-u^2}}{|z_2|}\in (0,\frac{\pi}{2}].
\]
If $u> 0$, then also
\[
\arccot\frac{\sqrt{|z_1|^2-u^2}}{u|z_2|} \in (0,\frac{\pi}{2}].
\]
Denote
\[
B(u)=\arccot\frac{\sqrt{|z_1|^2-u^2}}{u|z_2|}-u\arccot\frac{\sqrt{|z_1|^2-u^2}}{|z_2|}.
\]
One finds its graph in Figure~4. Obviously it is an odd function and $|B(u)|\leq \frac{\pi}{2}$.
 The inequalities $-\pi<\theta_1<\pi$ and $|u|<1$ imply $-\frac{3\pi}{2}<2\pi(p-uq)<\frac{3\pi}{2}$. Therefore, $-\frac{3}{4}-q<p<q+\frac{3}{4}$, or $-q\leq p\leq q$. 
The equation $\theta_1=B(u)$ corresponds to the choice $q=0$, $p=0$. We visualize possible choices of $p$ for $q=3$ in Figure~5. 
\vspace{10pt}
\begin{figure}\label{fig1}
\begin{center}
\epsfig{file=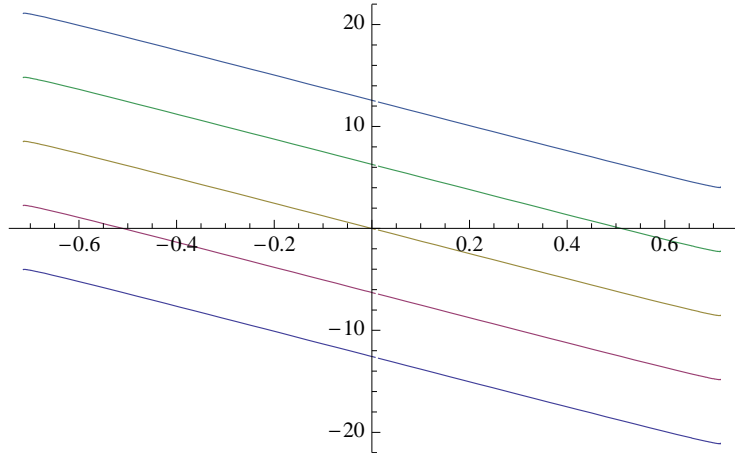,height=6cm}
\end{center}
\caption[]{Graph of the function $B(u)+2\pi(p-3u)$, $|z_1|=0.7$.}
\end{figure}

\item {\bf Case $\sigma_1<0$, $\sigma_2<0$.} The equations for $u$ and $\rho>0$ turn into
\[
\rho=\arccot\frac{\sqrt{|z_1|^2-u^2}}{|z_2|}+\pi (2q+1),\quad q\in \mathbb N\cup\{0\},
\]
\[
\theta_1=\arccot\frac{\sqrt{|z_1|^2-u^2}}{u|z_2|}-u\arccot\frac{\sqrt{|z_1|^2-u^2}}{|z_2|}+\pi((2p+1)-u(2q+1)),
\]
$p\in \mathbb Z$ with the same branch of $\arccot$ as in the previous case. The inequalities $-\pi<\theta_1<\pi$ and $|u|<1$ imply $-q-1\leq p\leq q+1$.

\item {\bf Cases $\sigma_1 \sigma_2=-1$.} These two cases come down to the previous ones changing $\rho\to -\rho$.
\end{itemize}

\item {\bf Carnot-Carath\'eodory distance.}  The length of a geodesic is calculated as $\rho \sqrt{1-u^2}$, where $u$ and $\rho$ are defined below. 
\begin{itemize}  
\item {\bf Let} $0<|\theta_1|\leq\frac{\pi}{2}(1-|z_1|)$. In the case $\sigma_1>0$, $\sigma_2>0$, there is a unique solution $u=u_0$ to the equation $B(u)=\theta_1$ (see Figure~4). The corresponding value for $\rho$ with $q=0$ we denote by $\rho_0$. All solutions to the equation $B(u)+2\pi(p-uq)=\theta_1$ with $q>0$ we denote by $u_{p,q}$ and the corresponding value of $\rho$ we denote by $\rho_{p,q}$. 
The function 
\[
\sqrt{1-x^2}\left(\arccot\frac{\sqrt{|z_1|^2-x^2}}{|z_2|}\right)
\]
is even and increases in $x\in[0,|z_1|]$, which
implies that
\[
\rho_0\sqrt{1-u_0^2}=\sqrt{1-u_0^2}\left(\arccot\frac{\sqrt{|z_1|^2-u_0^2}}{|z_2|}\right)\leq\frac{\pi}{2}|z_2|<2\pi q |z_2|\leq 
\]
\[
\leq \sqrt{1-u_{p,q}^2}\left(\arccot\frac{\sqrt{|z_1|^2-u_{p,q}^2}}{|z_2|}+2\pi q\right)=\rho_{p,q}\sqrt{1-u_{p,q}^2}.
\]
Similarly other choices of $\sigma_1$ and $\sigma_2$ do not give the minimizing geodesic.
So if $0<|\theta_1|\leq\frac{\pi}{2}(1-|z_1|)$, then the minimal length is 
\[
d=\sqrt{1-u_0^2}\left( \arccot\frac{\sqrt{|z_1|^2-u_0^2}}{|z_2|}\right),
\]
where $u_0$ is a unique solution to the equation $B(u)=\theta_1$.
 
\item {\bf Let} $|\theta_1|>\frac{\pi}{2}(1-|z_1|)$. Then there is no solution to the equation $B(u)=\theta_1$. The situation is even more complicated.
We visualize it in Figure~6. If $q=1$ we can guarantee a solution to the equation $B(u)+2\pi(p-u)=\theta_1$ for $|z_1|\geq 3/4$.

\vspace{0pt}
\begin{figure}
\begin{center}
\begin{pspicture}(0,0)(7,7)
\put(0,0){\epsfig{file=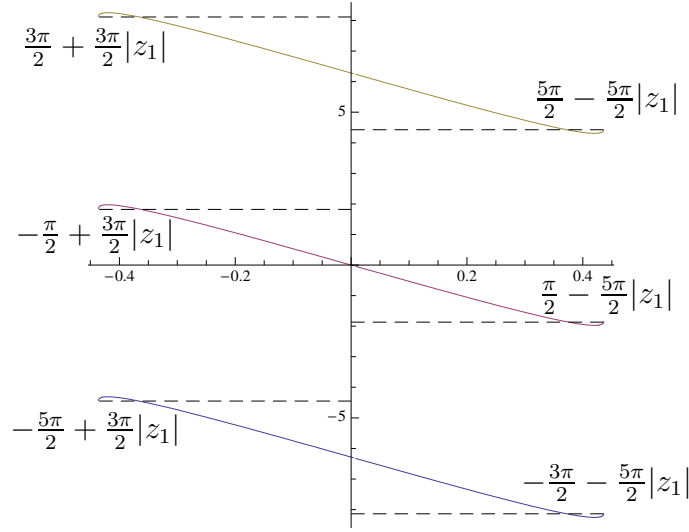,height=7cm}}
\psline[linecolor=black, linestyle=dashed, linewidth=0.1mm]{-}(0.15,6.8)(3.5,6.8)
\psline[linecolor=black, linestyle=dashed, linewidth=0.1mm]{-}(0.15,4.24)(3.5,4.24)
\psline[linecolor=black, linestyle=dashed, linewidth=0.1mm]{-}(0.15,1.69)(3.5,1.69)
\rput(0.1,6.4){$\frac{3\pi}{2}+\frac{3\pi}{2}|z_1|$}
\rput(0.1,3.9){$-\frac{\pi}{2}+\frac{3\pi}{2}|z_1|$}
\rput(0.1,1.3){$-\frac{5\pi}{2}+\frac{3\pi}{2}|z_1|$}
\psline[linecolor=black, linestyle=dashed, linewidth=0.1mm]{-}(3.5,5.3)(6.85,5.3)
\psline[linecolor=black, linestyle=dashed, linewidth=0.1mm]{-}(3.5,2.74)(6.85,2.74)
\psline[linecolor=black, linestyle=dashed, linewidth=0.1mm]{-}(3.5,0.19)(6.85,0.19)
\rput(6.9,5.7){$\frac{5\pi}{2}-\frac{5\pi}{2}|z_1|$}
\rput(6.9,3.1){$\frac{\pi}{2}-\frac{5\pi}{2}|z_1|$}
\rput(6.9,0.6){$-\frac{3\pi}{2}-\frac{5\pi}{2}|z_1|$}
\end{pspicture}
\end{center}
\caption[]{Graph of the function $B(u)+2\pi(p-u)$ for $p=-1,0,1$.}
\end{figure}

Therefore, in the case when  the equation $B(u)+2\pi(p-u)=\theta_1$ has a solution,  let us consider  $q=1$ and $q>1$. The function
\[
\sqrt{1-x^2}\left(\arccot\frac{\sqrt{|z_1|^2-x^2}}{|z_2|}+2\pi\right)
\]
attains its maximum in the interval $x\in [-|z_1|,|z_1|]$ at the point $x=0$. Moreover, we have an elementary trigonometric inequality
\[
\arccot\frac{\sqrt{1-x^2}}{x}< 2\pi x,\quad x\in (0,1].
\]
Hence, in the case $\sigma_1>0$, $\sigma_2>0$ and $q\geq 2$, we have the following chain of inequalities
\[
\rho_{p,1}\sqrt{1-u_{p,1}^2}=\sqrt{1-u_{p,1}^2}\left(\arccot\frac{\sqrt{|z_1|^2-u_{p,1}^2}}{|z_2|}\right)\leq \arccot\frac{|z_1|}{|z_2|}<
\]
\[
<2\pi q |z_2|\leq \sqrt{1-u_{p,q}^2}\left(\arccot\frac{\sqrt{|z_1|^2-u_{p,q}^2}}{|z_2|}+2\pi q\right)=\rho_{p,q}\sqrt{1-u_{p,q}^2}.
\]
Thus, $d=\min\limits_{p=-1,0,1}\rho_{p,1}\sqrt{1-u_{p,1}^2}$.

\item {\bf In general case}, we are able to give an algorithm of finding geodesics and the length. Let $q_m$ be a minimal non-negative integer for which
the equation $B(u)+2\pi(p-qu)=\theta_1$ has a solution. Then the distance is calculated as
\[
d=\min\left\{\rho_{p,q_m}\sqrt{1-u_{p,q_m}^2}\bigg| \,\,\,\mbox{among\ } p=\{-q_m,-q_m+1,\dots, -1,0,1,\dots, q_m-1, q_m\right\}.
\]
\end{itemize}
\end{itemize}


\begin{thebibliography}{99}

\bibitem{AS}
A.~Agrachev, Yu.~Sachkov, {\it Control theory from the geometric viewpoint}, 
Encyclopaedia of Math. Sci., 87. Control Theory and Optimization, II. Springer-Verlag, Berlin, 2004,~412 pp.

\bibitem{Anast}
Ch.~Anastopoulos, N.~Savvidou, {\it Quantum mechanical histories and the Berry phase}, Intern. J. Theor. Phys. {\bf 41} (2002), no. 3, 529--540.

\bibitem{Berry}
M. V. Berry, {\it Quantal phase factors accompanying adiabatic changes.} Proc. R. Soc. Lond. A 392 (1984), 45--57.

\bibitem{Ugo}
U.~Boscain, G.~Charlot, J.-P.~Gauthier, S.~Gu\'erin, H.-R.~Jauslin, {\it Optimal Control in laser-induced population transfer for two- and three-level quantum systems}, J. Math. Phys. {\bf 43} (2002),  2107--2132.

\bibitem{BosRossi}
U. Boscain, F. Rossi, {\it Invariant Carnot-Caratheodory metrics on $S^3, {\rm SO}(3), {\rm SL}(2)$, and lens spaces.} SIAM J. Control Optim. {\bf 47} (2008), no. 4, 1851--1878.

\bibitem{Bou}
D.~Bouwmeester, A.~Eckert, A.~Zeilinger, {\it The physics of quantum information}, Springer-Verlag, 2000.

\bibitem{CalinChang}
O.~Calin, D.-C.~Chang, {\it Sub-Riemannian geometry. General theory and examples}, Cambridge Univ. Press, 2009.

\bibitem{CalinChangMarkina1}
O.~Calin, D.-Ch.~Chang, I.~Markina, {\it Sub-Riemannian geometry on the sphere $S^3$}, Canadian J. Math. {\bf 61} (2009), no. 4, 721--739.


\bibitem{ChMV}
D. C. Chang, I. Markina, A. Vasil'ev, {\it Sub-Riemannian geodesics on the 3-D sphere.} Complex Anal. Oper. Theory {\bf 3} (2009), no. 2, 361--377.

\bibitem{Chow}
W.~L.~Chow. {\it Uber Systeme von linearen
partiellen Differentialgleichungen erster Ordnung}, Math. Ann.,
{\bf 117} (1939), 98-105.

\bibitem{HR}
A. Hurtado, C. Rosales,
{\it Area-stationary surfaces inside the sub-Riemannian
three-sphere}. Math. Ann. {\bf 340} (2008),  no. 3, 675--708.

\bibitem{Milnor}
J.~Milnor, {\it Morse Theory.} Annals of Math. Studies. {\bf 51}
Princeton University Press. 1973

\bibitem{GM1}
M. Godoy Molina, I. Markina,
{\it Sub-Riemannian geometry on parallelizable spheres}, Revista Matem. Iberoamericana (to appear), arXiv 0901.1406 (2009).

\bibitem{GM2}
M. Godoy Molina, I. Markina,
{\it Sub-Riemannian geodesics and heat operator on odd dimensional spheres},   arXiv 1008.5265 (2010).

\bibitem{Montgomery}
R.~Montgomery,  {\it A tour of subriemannian geometries, their geodesics and applications.}
Mathematical Surveys and Monographs, {\bf 91}. American Mathematical Society, Providence, RI, 2002. 259 pp.

\bibitem{Mosseri}
R.~Mosseri, R.~Dandoloff, {\it Geometry of entangled states, Bloch spheres and Hopf fibrations}, J. Phys. A: Math. Gen. {\bf 34} (2001), 10243--10252.

\bibitem{Pancharatnam}
S.~Pancharatnam, {\it Generalized theory of interference, and its applications. Part I. Coherent pencils.} Proc. Indian Acad. Sci. A {\bf 44} (1956), 247--262.

\bibitem{qubit}
T.~Radtke, S.~Fritzsche,
{\it Simulation of n-qubit quantum systems.} Computer Physics Communications  {\bf 179} (2008), no.~9, 647--664.

\bibitem{Rashevsky}
P.~K.~Rashevski{\u\i}, {\it About connecting two points of complete nonholonomic space by admissible curve}, Uch. Zapiski Ped. Inst. K.~Liebknecht {\bf 2} (1938), 83--94.

\bibitem{Strichartz}
R.~S.~Strichartz,  {\it Sub-Riemannian geometry},  J. Differential
Geom. {\bf 24} (1986) 221--263; Correction, ibid. {\bf 30}
(1989) 595-596.

\bibitem{Urbantke1}
H.~Urbantke, {\it Two-level quantum systems: states, phases, and holonomy}, Amer. J.~Phys. {\bf 59} (1991), no. 6, 503--509.

\bibitem{Urbantke2}
H.~Urbantke, {\it The Hopf fibration--seven times in physics}, J.~Geom.~Phys. {\bf 46} (2003), 125--150.

\end{thebibliography}
\end{document}